\documentclass[]{article}

\usepackage{makeidx}  

\usepackage{amsthm} 

\usepackage{latexsym}
\usepackage[pdftex]{graphics}
\usepackage{xypic}
\usepackage{amsmath}
\usepackage{amssymb}
\xyoption{all} 
\usepackage{mathrsfs}
\usepackage{eufrak}

\newtheorem{theorem}{Theorem}[section]

\newtheorem{lemma}[theorem]{Lemma}

\newtheorem{definition}[theorem]{Definition}

\newtheorem{proposition}[theorem]{Proposition}

\newtheorem{corol}[theorem]{Corollary}

\newtheorem{examples}[theorem]{Examples}

\newtheorem{remark}[theorem]{Remark}

\newtheorem{notation}[theorem]{Notation}

\newtheorem{convention}[theorem]{Convention}

\newcommand{\C}{{\mathcal C}}

\newcommand{\F}{{\mathcal F}}

\newcommand{\M}{{\mathcal M}}

\newcommand{\W}{{\mathcal W}}
\newcommand{\X}{{\mathcal X}}

\newcommand{\decode}{\rhd}
\newcommand{\code}{\lhd}

\newcommand{\dc}{{\lhd\hspace{-0.148em}\rhd}}

\begin{document}

\author{Peter M. Hines}
 
 \title{Coherence and strictification for self-similarity}

\maketitle
\begin{abstract}
{\em 
This paper studies questions of coherence  and strictification related to self-similarity --- the identity $S\cong S\otimes S$ in a (semi-)monoidal category. Based on Saavedra's theory of units, we first demonstrate that strict self-similarity cannot simultaneously occur with strict associativity -- i.e. no monoid may have a strictly associative (semi-)monoidal tensor, although many monoids have a semi-monoidal tensor associative up to isomorphism. We then give a simple coherence result for the arrows exhibiting self-similarity and use this to describe a `strictification procedure' that gives a semi-monoidal equivalence of categories relating strict and non-strict self-similarity, and hence monoid analogues of many categorical properties. Using this, we characterise a large class of diagrams (built from the canonical isomorphisms for the relevant tensors, together with the isomorphisms exhibiting the self-similarity) that  are guaranteed to commute.
 }

\end{abstract}

\section{Introduction}\label{intro}
An object $S$ in a semi-monoidal category $(\mathcal C,\otimes)$ is {\em self-similar} when it satisfies the identity $S\cong S\otimes S$. In a 1-categorical sense, self-similar objects are simply pseudo-idempotents and thus share many categorical properties with unit objects as characterised by Saavedra (see Section \ref{sav_units}); they also provide particularly well-behaved examples of split idempotents (see Section \ref{splits}). The most familiar non-unit example of a self-similar object is undoubtedly the natural numbers in the monoidal categories $({\bf Set},\times)$ and $({\bf Set},\uplus)$, as illustrated by Hilbert's parable of the 'Grand Hotel' (see \cite{NY} for a good exposition in a general context).  Topologically, self-similarity is clearly seen in the Cantor set \& other fractals \cite{PH98,blatant}; algebraically, it is very closely connected with Thompson's groups (see Section \ref{Thompson}), the polycyclic monoids \cite{PH98,MVL,PH99}, and finds applications to algebraic models of tilings \cite{KeLa}.

In computer science, self-similarity plays a key role in categorical models of untyped systems such as the C-monoids of \cite{LS}  (single-object Cartesian closed categories without unit objects modelling untyped lambda calculus -- see \cite{WHS} for a survey). It is particularly heavily used in Girard's Geometry of Interaction program  \cite{GOI,GOI2} where, as well as being a key feature of the `dynamical algebra' it is used together with compact closure to construct monoids isomorphic to their own endomorphism monoid \cite{AHS,PH98,PH99}.   More recently, it has found applications in linguistic and grammatical models \cite{PH13c}, categorical models of quantum mechanics \cite{PH14a}, and -- via the close connections with Thompson's groups - is relevant to  cryptography and cryptanalysis \cite{SU,PH14b}, and homotopy idempotents (see Section \ref{Thompson}). 

The aim of this paper is to give coherence results and a strictification procedure for self-similarity, and so to relate the isomorphisms exhibiting self-similarity with canonical isomorphisms for the relevant (semi-)monoidal structures.  The motivation for this  is the observation (Theorem \ref{no_sim_strict}), based on Saavedra's theory of units (See Section \ref{sav_units}), that strict self-similarity and strict associativity are mutually exclusive --- either one or the other of these properties must be up to non-trivial isomorphism.

\section{Categorical preliminaries}

We refer to \cite{MCL} for the theory of monoidal categories. We work with a slight generalisation that satisfies all the axioms for a monoidal category except for the existence of a unit object; following \cite{K}, we refer to these as {\em semi-monoidal}.
\begin{definition}
A {\bf semi-monoidal category}  is a category $\C$ with a functor $\_\otimes \_:\C\times \C\rightarrow \C$ that is associative up to an object-indexed family of natural isomorphisms $\tau_{X,Y,Z}:X\otimes (Y\otimes Z)\rightarrow (X\otimes Y)\otimes Z$ satisfying MacLane's pentagon condition
\[  ( \  \tau_{W,X,Y}\otimes 1_Z)  \  \tau_{W,X\otimes Y,Z}  (1_W \otimes  \  \tau_{X,Y,Z}) \ = \   \  \tau_{W\otimes X,Y,Z}  \  \tau_{W,X,Y\otimes Z} \] 
A functor between semi-monoidal categories that (strictly) preserves the tensor is a {\bf (strict) semi-monoidal functor}.  We assume the obvious definition of {\bf semi-monoidal equivalence of categories}.
\end{definition}
(The above definition  differs from that of \cite{K} in that we do not assume strict associativity -- see Theorem \ref{no_sim_strict} for the motivation for this). When we do consider unit objects, we will use Saavedra's characterisation, rather than MacLane \& Kelly's original definition -- see Section \ref{sav_units} below.
\begin{remark}   
Any category $\mathcal C$ may be given a (degenerate) semi-monoidal tensor by fixing some object $X\in Ob(\C)$, and and defining $A\otimes B=X$, $f\otimes g=1_X$ for all objects $A,B\in Ob(\C)$ and arrows $f,g\in Arr(\C)$. In the theory of monoidal categories, standing assumptions such as (monoidal) well-pointedness are used to exclude such pathologies (see, for example, \cite{AH}); in the absence of a unit object we will instead use the assumption of {\em `faithful objects'} given below.
\end{remark}

\begin{definition}\label{fo} A monoidal category $(\C,\otimes,I)$ is called {\bf monoidally well-pointed} when  
$f=g\in \C(A\otimes X,B\otimes Y) \mbox{ iff } \forall \ a\in \C(I,A) , x\in \C(I,B), \ \  f(a\otimes x) = g (a\otimes x)$. 
A consequence of monoidal well-pointedness is that for arbitrary $A\in Ob(\C)$, where $A$ is not a strict retract of the unit object, the functors $(A\otimes \_ ),(\_ \otimes A):\C\rightarrow \C$ are faithful.  Motivated by this, we say that an object $A$ of a semi-monoidal category $(C,\otimes )$ is {\bf faithful} when the functors $A\otimes \_, \_ \otimes A :\C\rightarrow \C$  are faithful. 
\end{definition}

\begin{convention}\label{assume_fo}{\bf Objects of semi-monoidal categories are faithful} 
Semi-monoidal categories in the literature commonly arise as semi-monoidal subcategories of monoidally well-pointed categories 
  and have faithful objects. To avoid repeatedly emphasising this, we will instead indicate when an object of a semi-monoidal category is {\em not} faithful.
\end{convention}

\subsection{Saavedra's theory of units}\label{sav_units}
MacLane's original presentation of the theory of coherence for monoidal categories gave a single coherence condition for associativity (the Pentagon condition) and four coherence conditions for the units isomorphisms. Three of these four axioms were shown to be redundant in \cite{MK}, leaving a single coherence condition expressing the relationship between the units isomorphisms, and associativity. In \cite{NS} an alternative characterisation of units objects was given that required no additional coherence conditions (see \cite{K} for a comprehensive study of this, and \cite{JK} for extensions of this theory).  We follow this approach; 
Definition \ref{SU} and Theorem \ref{SunitResults} below are  taken from \cite{K,JK}.
\begin{definition}\label{SU}
A {\bf (Saavedra) unit} $U$ in a semi-monoidal category $(\C,\otimes , \tau)$ is a cancellable pseudo-idempotent i.e. there exists an isomorphism $\alpha:  U \otimes U \rightarrow  U$, and the functors $(U\otimes \_ ), (\_ \otimes U) : \C \rightarrow \C$ are fully faithful.
\end{definition}  

We also refer to \cite{K} for the following key results:
\begin{theorem}\label{SunitResults}$ $ \\
\begin{enumerate}
\item Saavedra units are units in the sense of MacLane / Kelly, and thus a semi-monoidal category with a (Saavedra) unit is a monoidal category.
\item Saavedra units are idempotents; i.e. $\alpha \otimes I_U \ = \ 1_U \otimes \alpha$ 
\item The functors $U\otimes \_$ and $\_ \otimes U$ are equivalences of categories.
\end{enumerate}
\end{theorem}

\section{Self-similar objects and structures}
The theory of self-similarity is precisely the theory of pseudo-idempotents in (semi-)monoidal categories. The difference in terminology arises for historical reasons -- in particular, differing conventions in computer science, mathematics, and pure category theory.

\begin{definition}
An object $S$ in a semi-monoidal category $(\C,\otimes)$ is called {\bf self-similar} when $S\cong S\otimes S$. Making the isomorphism exhibiting this self-similarity explicit, a {\bf self-similar structure} is a tuple $(S,\code)$ consisting of an object $S\in Ob(\C)$, and an isomorphism $\code:S\otimes S \rightarrow S$ called the {\bf code} isomorphism. We denote its inverse by $\code^{-1}=\decode : S\rightarrow S\otimes S$ and refer to this as the {\bf decode} isomorphism.

A {\bf strictly self-similar object} is an object $S$ such that $(S,1_S)$ is a self-similar structure. 
The endomorphism monoid of a strictly self-similar object is thus itself a semi-monoidal category with a single object -- i.e. it is a {\bf semi-monoidal monoid}.
\end{definition}

\begin{examples}Examples of non-strict self-similarity are discussed in Section \ref{intro}. 
Strict examples include Thompson's group F (see \cite{KB} for the semi-monoidal tensor and associativity isomorphism of this group, and Section \ref{Thompson} for the explicit connection with strict self-similarity), and the group of bijections on the natural numbers, with the semi-monoidal tensor given by 
\[ (f\star  g) (n) \ = \ \left\{\begin{array}{lr} 
 \vspace{0.4em}
2.f\left(\frac{n}{2}\right) & n \mbox{ even,} \\
2.g\left(\frac{n-1}{2}\right) + 1  & n \mbox{ odd,} 
\end{array} \right. \]
and canonical associativity and symmetry isomorphisms respectively given by 
\[   \tau (n) =\left\{ \begin{array}{lr}
 \vspace{0.4em}
2n &  \ \ \ n\ (mod \ 2)=0, \\
 \vspace{0.4em}
n+1 &  \ \ \ n\ (mod \ 4)=1,  \\	
\frac{n-3}{2} &  \ \ \  n\ (mod \ 4)=3.  \\	
\end{array}\right. \ \ \ , \ \ \ 
\sigma (n)  \ = \ \left\{\begin{array}{lr} 
 \vspace{0.4em}
n+1 & n \mbox{ even,} \\
n-1  & n \mbox{ odd.} 
\end{array} \right. \] 
An elementary arithmetic proof that the above data specifies a semi-monoidal monoid 
is given in \cite{PH13b}. More generally, it arises from a special case of a large class of representations of Girard's dynamical algebra (viewed as the closure of the two-generator polycyclic monoid \cite{NP} under the natural partial order of an inverse semigroup) as partial functions, given in \cite{PH98,MVL}.  Readers familiar with the Geometry of Interaction program will recognise the $(\_ \star \_)$ operation as Girard's model (up to Barr's $l_2:{\bf pInj}\rightarrow {\bf Hilb}$ functor -- see \cite{MB,CH}) of the (identified) multiplicative conjunction \& disjunction of \cite{GOI,GOI2}.
\end{examples}

\subsection{Self-similarity as idempotent splitting}\label{splits}
Self-similar structures are a special case of {\em idempotent splittings}. We refer to \cite{FS,LS,PS,RS} for the general theory and reprise some basic properties below:
\begin{definition}
An idempotent $e^2=e\in \C(A,A)$ {\bf splits} when there exists some $B\in Ob(\C)$ together with arrows $f\in \C(A,B)$, $g\in \C(B,A)$ such that $e=gf$ and $fg=1_B$. We refer to the pair $(f,g)$ as a {\bf splitting} of the idempotent $e^2=e$.
\end{definition}

\begin{remark}\label{SSasSplit}
We may characterise self-similar structures in terms of splittings of identities: a self-similar structure $(S,\code)$ uniquely determines, and is uniquely determined by, an isomorphism $\code$ such that $(\code,\decode)$ is a splitting of $1_S$ and $(\decode,\code)$ is a splitting of $1_{S\otimes S}$.
\end{remark}
This characterisation allows us to use standard results on idempotent splittings, such as their uniqueness up to unique isomorphism:
\begin{lemma} Given an idempotent $e^2=e\in \C(A,A)$ together with splittings $(f\in \C(A,B), g\in \C(B,A))$ and $(f'\in \C(A,B'),g'\in \C(B',A))$, then there exists a unique isomorphism $\phi:B\rightarrow B'$ such that the following diagram commutes:
\[ \xymatrix{ 
A \ar[r]^f \ar[d]_{f'}				& B \ar[d]^g \ar[dl]|\phi				\\
B'	\ar[r]_{g'}					&	A			\\
}
\]
\end{lemma}
\begin{proof}
This is a standard result of the theory of idempotent splittings. See, for example, \cite{PS}.
\end{proof}

\begin{corol} \label{unique}
Self-similar structures at a given self-similar object are unique up to unique isomorphism
\end{corol}
\begin{proof} The proof of this is somewhat simpler than the general case, as the splittings have both left and right inverses. Given a self-similar structure $(S,\code)$, and an isomorphism $U:S\rightarrow S$, then  
$(S,U\code)$ is also a self-similar structure. Conversely, let $(S,\code')$ be a self-similar structure, and define $U=\code'\decode$. Then  $(S,U\code) = (S,\code')$ and $U=\code'\decode$ is the unique isomorphism satisfying this condition.
\end{proof}

\begin{remark}{\em Uniqueness of self-similar structures}\\
Corollary \ref{unique} provides an illustration of the distinction between `unique up to unique isomorphism', and `actually unique'.  If a self-similar structure at some object $S\in Ob(\C)$ is actually unique, then the only isomorphism from $S$ to itself is the identity. Theorem \ref{no_sim_strict} below then shows that $S$ must be the unit object for the tensor given in Theorem \ref{untyped}.
\end{remark}
The above characterisation of self-similar structures as idempotent splittings also gives an abstract categorical characterisation in terms of limits and colimits of diagrams:
\begin{proposition}
Given a self-similar structure $(S,\code)$ of a semi-monoidal category $(\C,\otimes)$, then 
\begin{enumerate}
\item $\xymatrix{ S \ar[r]^\decode & S \otimes S}$ is a colimit of 
$\xymatrix{ S \ar@(ur,ul)[]_{1_S}}$ 
and a limit of $\xymatrix{ S\otimes S \ar@(ul,ur)[]^{\ \ 1_{S\otimes S}}}$,
\item $\xymatrix{ S\otimes S \ar[r]^\code & S }$ is a limit of 
$\xymatrix{ S \ar@(dr,dl)[]^{1_S}}$ 
and a colimit of $\xymatrix{ S\otimes S \ar@(dl,dr)[]_{1_{S\otimes S}}}$,
\end{enumerate}
and these two conditions characterise self-similar structures at $S\in Ob(\C)$.
\end{proposition}
\begin{proof}
A standard result on splitting idempotents is that a splitting of an idempotent $e^2=e\in \C(A,A)$ is precisely  a pair $(f\in \C(A,B),g\in \C(B,A))$ such that $\xymatrix{ A \ar[r]^f & B }$ and $\xymatrix{ A &   B \ar[l]_g  }$ are respectively a limit and colimit of the diagram 
 $\xymatrix{ A \ar@(ur,dr)[]^{e}}$. The result then follows from the characterisation of self-similar structures given in Remark \ref{SSasSplit}.
 \end{proof}

\section{Strict self-similarity and strict associativity}
Unit objects are special cases of self-similar objects -- the distinction being that for an arbitrary self-similar object $S$, the functors $S\otimes \_$ and $\_ \otimes S$ need not be fully faithful. We now describe how Saavedra's characterisation relates strictness for both associativity and self-similarity.

\begin{lemma} \label{strict_prelim}Let $(\M,\star )$ be semi-monoidal monoid. The endomorphisms $(\_ \star 1)$ and $(1 \star \_)$  are injective, and are isomorphisms precisely when  the unique object of $M$ is the unit. 
\end{lemma}
\begin{proof}
Functoriality implies that $\_ \star 1$ and $1 \star \_$ are monoid homomorphisms, 
and injectivity follows from the assumption (Convention \ref{assume_fo}) of faithful objects. These homomorphisms are isomorphisms precisely when they are fully faithful, in which case the unique object of $\M$ satisfies Saavedra's characterisation of a unit object.
\end{proof}

\begin{theorem}\label{no_sim_strict}
Let $(\M,\star , \tau)$ be a semi-monoidal monoid. Then $\_ \star \_$ is strictly associative if and only if the unique object of $M$ is the unit for $\_\star \_$.
\end{theorem}
\begin{proof}
$(\Rightarrow)$ 
From Lemma \ref{strict_prelim} above, and strict associativity, $1 \star (\_ \star 1) ,(1 \star \_) \star 1 : \M \rightarrow \M$ are identical injective monoid homomorphisms. Let us denote this monoid embedding by $\eta : \M \hookrightarrow \M$. By injectivity, for every arrow $F\in \eta(\M)$, there exists a unique arrow $f\in M$ satisfying $F=1\star f \star 1$.

Define a (strict) semi-monoidal tensor $(\_ \odot \_ ) : \eta (\M) \times \eta (\M) \rightarrow \eta(\M)$ by $ F \odot G = 1 \star (f\star g) \star 1$, for all $F=1\star f \star 1, G = 1 \star  g \star 1$.
It is immediate that this is well-defined, and a semi-monoidal tensor. By definition, $\eta (f\star g) = \eta (f) \odot \eta (g)$, so $(\M,\star) \cong (\eta(\M),\odot)$. 

Finally $(1\odot F) = 1 \star (1\star f) \star 1 = 1\star 1 \star f \star 1 = 1\star f \star 1 = F$, for  arbitrary $ F=1\star f \star 1$,
and hence $1\odot \_ : \eta (\M)\rightarrow \eta(\M)$ is the identity isomorphism. Similarly, $\_ \odot 1= Id_{\eta(\M)}$, and thus by Lemma \ref{strict_prelim}, the unique object of $(\M,\star) \cong (\eta(\M),\odot)$ is the unit object.\\
$(\Leftarrow)$ It is a standard result of monoidal categories that the endomorphism monoid of a unit object is an abelian monoid, and the tensor at this object coincides (up to isomorphism) with this strictly associative composition.
\end{proof}

\begin{remark}{\bf No simultaneous strictification}
This paper is about strictification and coherence for self-similarity and its interaction with associativity. From Theorem \ref{no_sim_strict}, strictifying the associativity of a semi-monoidal monoid will result in non-strict self-similarity; conversely, strictifying self-similarity in a strictly associative setting will give a monoid with a non-strict semi-monoidal tensor (Proposition \ref{not_strict}).
\end{remark}

\begin{examples}
\label{infinitematrices}{\bf Finite and infinite matrices}
An illustrative example is given by infinitary matrix categories. Countable matrices over a $0$-monoid $R$ enriched with a suitable (partial, infinitary) summation\footnote{See  \cite{PH13a} for a suitable unification of various notions of summation from theoretical computer science and analysis, and some associated category theory.}  $\mathcal R$ form a category $\bf Mat_\mathcal R$, with $Ob(Mat_\mathcal R)=\mathbb N \cup \{ \infty \}$. Hom-sets  $\bf {Mat_\mathcal R} (a,b)$ are subsets\footnote{The set of allowable matrices is generally restricted by some summability condition, the details of which do not affect the substance of the following discussion.} of the set of functions ${\bf Fun}([0,b)\times [0,a), \mathcal R)$, and composition is given by the Cauchy product. 

Assuming  technical conditions on summation are satisfied, the subcategory of finite matrices has a strictly associative monoidal tensor, denoted $\_ \oplus \_$. On objects it is simply addition; given arrows,  $m\in Mat_R(a,b)$, $n\in Mat_R(p,q)$,  with $a,b,p,q<\infty$, it is the familiar block diagonal formula $ \left(m\oplus n\right)(x,y)  = \left\{ \begin{array}{lr} 
m(x,y) & x<a \ , \ y<b \\
n(x-a,y-b) & x\geq a \ , \ y\geq b \\
0 & \mbox{otherwise.}
\end{array} \right. 
 $
 
Although this definition cannot be extended to {infinite} matrices, the endomorphism monoid ${\bf Mat_\mathcal R}(\infty, \infty)$ can be given a {non-strict} semi-monoidal tensor (again, assuming technicalities on summation), such as 
$(A\oplus_c  B)(x,y) = \left\{ \begin{array}{lr} 
\vspace{0.4em}
A\left(\frac{x}{2},\frac{y}{2}\right) & x,y \ \mbox{ even,}	\\  \vspace{0.4em}
B\left(\frac{x-1}{2},\frac{y-1}{2}\right) & x,y \ \mbox{ odd,}	\\ 
0 & \mbox{ otherwise.}
\end{array}\right.
$

 \noindent
 This is the familiar `interleaving' of infinite matrices, determined by the Cantor pairing $c:\mathbb N \uplus \mathbb N\rightarrow \mathbb N$ given by $c(n,i)=2n+i$ (and used to great effect in modelling the structural rules of linear logic \cite{GOI,GOI2}). Any such isomorphism $\code: \mathbb N \uplus \mathbb N \rightarrow \mathbb N$ will determine a (non-strict) semi-monoidal tensor on this monoid.   However, by Theorem \ref{no_sim_strict}, no strict semi-monoidal tensor on ${\bf Mat_\mathcal R} (\infty,\infty)$ may  exist.
\end{examples}

\section{The group of canonical isomorphisms}\label{Thompson}
In a semi-monoidal monoid, the isomorphisms canonical for associativity are closed under composition, tensor, and inverses, and thus form a group with a semi-monoidal tensor. As demonstrated in \cite{FL}, in the free case this group is the well-known Thompson group $\mathcal F$ (see \cite{CFP} for a non-categorical survey). An algebraic connection between this group and associativity laws is well-established (see \cite{MBr} for a more categorical perspective), and the tensor was given in \cite{KB} -- although not in categorical terms. An explicit  connection with semi-monoidal monoids was observed in \cite{MVL07} where $\mathcal F$ is given in terms of canonical isomorphisms and single-object analogues of projection  / injection arrows for the tensor. 

An interesting connection between Thomson's group $\mathcal F$ and the theory of idempotent splittings is given in \cite{BG}, in the context of (connected, pointed) CW complexes. The {\em  unsplittable} homotopy idempotents of a CW complex $K$ are characterised by the fact that they give rise to a copy of $\mathcal F$ in $\pi_1(K)$. The categorical interpretation of this is unfortunately beyond the scope of this paper.

\section{A simple coherence result for self-similarity}
We now give a simple result that guarantees commutativity for a class of diagrams built inductively from the code / decode isomorphisms of a self-similar structure, and the relevant semi-monoidal tensor. This is modelled very closely indeed on MacLane's original presentation of his coherence theorem for associativity (briefly summarised in Section \ref{MCL_coherence}), in order to describe the interaction of self-similarity and associativity.

\subsection{Coherence for associativity - the unitless setting}\label{MCL_coherence}
We first briefly reprise some basic definitions and results on coherence for associativity, taken from MacLane's original presentation \cite{MCL}, in the semi-monoidal setting. This is partly to fix notation and terminology, and partly to ensure that the absence of a unit object does not lead to any substantial difference in theory.   We also restrict ourselves to the monogenic case, as this suffices to describe the interaction of associativity and self-similarity. 

For more general, structural approaches to coherence, we refer to \cite{MK2,AP,JS} -- these inspired the approach taken  in Section \ref{self_strict} onwards.

\begin{definition}\label{Wdef}
 The set $Tree$ of {\bf free non-empty binary trees} over some symbol $x$ is is inductively defined by: 
$x\in Tree $, and for all $u,v\in Tree $ then $(u\Box v)\in Tree $.
The {\bf rank} of a tree $t\in Tree$ is  the number of occurrences of the symbol $x$ in $t$, or equivalently, the number of leaves of $t$.

We denote (the unitless version of) MacLane's {\bf monogenic monoidal category} by $(\W,\Box)$. This is defined by: $Ob(\W)=Tree$ and there exists a unique arrow $(t\leftarrow s)\in \mathcal\W(s,t)$ iff $rank(s)=rank(t)$. Composition is determined by uniqueness. 
Given $p,q\in Ob(\W)$, their semi-monoidal tensor is $p\Box q$; the tensor of arrows is again determined by uniqueness. 
\end{definition}

\begin{remark} MacLane's definition \cite{MCL} included the empty tree as a unit object, giving a monoidal, rather than semi-monoidal, category. 
Applying the common technique of adjoining a strict unit to a semi-monoidal category will recover MacLane's original definition, and MacLane's original theory in the exposition below.
\end{remark}

\begin{definition}\label{Wsub} Given an object $A$ of a semi-monoidal category $(\C,\otimes ,  \tau_{\_,\_,\_})$, MacLane's  
{\bf associativity substitution functor} 
$\W Sub_{A\in Ob(\C)}  :(\W,\Box) \rightarrow (\C,\otimes)$ is defined inductively below. When the context is clear, we elide the subscript on $\W Sub_\_$.
\begin{itemize}
\item {\bf (Objects)} For all $u,v\in Ob(\W)$,
\begin{itemize} 
\item $\W Sub(u\Box v) = \W Sub(u) \otimes \W Sub(v)$
\item $\W Sub(x) =A\in Ob(\C)$.
\end{itemize}
\item {\bf (Arrows)} Given $a,b,c,u,v\in Tree $, where $rank(u)=rank(v)$, 
\begin{itemize}
\item $\W Sub(a\leftarrow a) = 1_{\W Sub(a)} \in \C(\W Sub(a),\W Sub(a))$.
\item $\W Sub((a\Box b)\Box c\leftarrow a\Box(b\Box c))= \  \tau_{\W Sub(a),\W Sub(b),\W Sub(c)}$
\item $\W Sub(a\Box v \leftarrow a\Box u)=1_{\W Sub(a)} \otimes \W Sub(v\leftarrow u)$
\item $\W Sub(b\Box u \leftarrow a\Box u)=\W Sub(b\leftarrow a) \otimes 1_{\W Sub(u)}$
\end{itemize}
\end{itemize}
\end{definition}

\begin{remark}\label{Pentagon_conditions}It is non-trivial that $\W Sub:(\W,\Box)\rightarrow (\C_A,\otimes)$ is a semi-monoidal functor. This is a consequence of MacLane's Pentagon condition \cite{MCL}. 
\end{remark}

\begin{remark} As $\W$ is posetal, all diagrams over $\W$ commute, so all (canonical) diagrams in $\C$ that are the image of a diagram in $\W$ are guaranteed to commute. Naturality and substitution are then used \cite{MCL} to extend this to the general setting.
\end{remark}

\subsection{A preliminary coherence result for self-similarity}
We now exhibit a class of diagrams based on identities, tensors, and the code / decode maps for a self-similar structure that are guaranteed to commute. This is based on a substitution functor from a posetal monoidal category that contains MacLane's $(\W , \Box )$ as a semi-monoidal subcategory.

\begin{definition}\label{Xdef}
The {\bf monogenic self-similar category} $(\X,\Box)$ was defined in \cite{PH98} as follows:
\begin{itemize}
\item {\bf Objects} $Ob(\X)=Tree$
\item {\bf Arrows} There exists unique $(b \leftarrow a)\in \X(a,b)$ for all $a,b\in Ob(\X)$.
\item {\bf Composition} This is determined by uniqueness.
\item {\bf Tensor} Given $u,v\in Ob(\X)$, their tensor is the binary tree $u\Box v$. The definition on arrows again follows from uniqueness.
\item {\bf Unit object} All objects $e\in Ob(\X)$ are unit objects;  the unique arrow $(e\leftarrow e\Box e)$ is an isomorphism, and the functors $(e\Box \_),(\_ \Box e):\X\rightarrow \X$ are fully faithful. 
\end{itemize}
\end{definition}

\begin{remark}
Abstractly, $(\X,\Box)$ may be characterised as the {\em free monogenic indiscrete monoidal category}. Thus, it is monoidally equivalent to the terminal monoidal category --- in the semi-monoidal setting, it is more interesting.
\end{remark}

\begin{definition}\label{Xsub}
Let $(S,\code)$ be a self-similar structure of a semi-monoidal category $(\C,\otimes )$. We define $\X Sub_\dc : (X,\Box) \rightarrow (\C,\otimes)$, the  {\bf self-similarity substitution functor}, inductively by, for  all $u,v,p,q\in Ob(\X )$:
\begin{itemize}
\item $\X Sub(x)=S$, and $\X Sub(u\Box v) = \X Sub(u) \otimes \X Sub(v)$.
\item $\X Sub(x\leftarrow u\Box v) = \code( \X Sub(x\leftarrow u) \otimes \X Sub(x \leftarrow v))$
\item $\X Sub(u\Box v\leftarrow x) = \decode (\X Sub(u \leftarrow x) \otimes \X Sub(v \leftarrow x))$
\item $\X Sub(u\Box v \leftarrow p\Box q) = \X Sub(u\Box v \leftarrow x) \X Sub(x \leftarrow p \Box q)$
\end{itemize}
(We again omit the subscript when the context is clear).
\end{definition}

\begin{remark}\label{Xconditions} In stark contrast to MacLane's substitution functor, it is immediate that $\X Sub : (X,\Box) \rightarrow (\C,\otimes)$ is a strict semi-monoidal functor -- no coherence conditions are needed to ensure functoriality. 
\end{remark}  

We may now give a preliminary coherence result on self-similarity. 

\begin{lemma}\label{selfsim-coherence} Let $(S,\code)$ be a self-similar structure of a semi-monoidal category $(\C,\otimes )$, and let $\X Sub: (X,\Box)\rightarrow (\C,\otimes)$ be as above. Then every diagram  over $\mathcal C$ of the form $\X Sub (\mathfrak D)$, for some diagram $\mathfrak D$ over $\X$, is guaranteed to commute. 
\end{lemma}
\begin{proof} As $(\X,\Box)$ is posetal, $\mathfrak D$ commutes; by functoriality so does $\X Sub(\mathfrak D)$.
\end{proof}

\begin{remark}\label{required_coherence}
The diagrams predicted to commute by Lemma \ref{selfsim-coherence} are `canonical for self-similarity', with arrows built from code / decode isomorphisms, identities, and the tensor. The more important question is about diagrams that are `canonical for self-similarity \& associativity' --- when may these be guaranteed to commute? 
This follows as a special case of a more general result (Section \ref{full_coherence}).
\end{remark}

\begin{remark}\label{doesitfactor}{\em \bf Does $\W Sub$ factor through $\X Sub$?}
There is an immediate semi-monoidal embedding $\iota : (\W,\Box) \rightarrow (\X,\Box)$.
An obvious question is whether, or under what circumstances,  the above substitution functors will factor through this embedding -- i.e. when does the diagram of Figure \ref{killer_inclusion} commute?
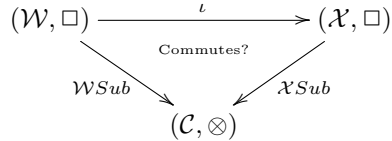
\begin{figure}[b]\caption{Under what circumstances does the following diagram commute?}\label{killer_inclusion}
\[ 
\xymatrix{
(\W,\Box) \ar[rr]^\iota	 \ar[dr]_{\W Sub} 	&	\ar@{}[d]|>>>>>>>>{\mbox{\tiny Commutes?}}		& (\X,\Box) \ar[ld]^{\X Sub} \\
								& (\C ,\otimes)	&
	}
\]	
\end{figure}

It is immediate that this can only commute under very special conditions. 
Functoriality of $\W Sub$ requires coherence conditions (i.e. MacLane's pentagon),  whereas none are required for the functoriality of $\X Sub$.  Further, commutativity of this diagram 
would give a decomposition of canonical (for associativity) isomorphisms  of $(\C,\otimes)$ into `more primitive' operations built from $\{ \code , \decode , \otimes\}$; in particular, $\tau_{S,S,S} = (\decode \otimes 1_S)(1_S \otimes \code)$. 
A slight generalisation of Isbell's argument on the skeletal category of sets \cite{MCL} would show that when $(\C,\otimes)$ admits projections, $S$ is the terminal object. 
Instead of giving a direct proof of this, we will give a more general result in Corollary \ref{inclusion_kills}.
\end{remark}

\section{Strictification for self-similarity}\label{self_strict}
We first describe a strictification procedure for self-similarity that gives a semi-monoidal equivalence between a monoid and a monogenic category, then use this to give a coherence theorem that answers the question posed in Remark \ref{required_coherence} in a more general setting.  This strictification procedure generalises the `untyping' construction of \cite{PH98,PH99} (and indeed, corrects it in certain cases -- see Remark \ref{inductive_def}). 

\begin{definition}\label{ins_def} Given a semi-monoidal category $(C,\otimes ,\tau_{\_,\_,\_})$ and arbitrary $S\in Ob(\C)$, the semi-monoidal category {\bf freely generated by $S$}, denoted $\mathcal F_{S}$, is defined analogously to the usual monoidal definition. The assignment $Inst:Tree\rightarrow Ob(\C)$ is defined inductively by $Inst(x)=S\in Ob(\C)$, and $Inst(p\Box q) = Inst(p) \otimes Inst(q)$ and based on this, objects and arrows are given by 
$Ob(\F_S)=Tree$, and $\F_S(u,v)= \C(Inst(u),Inst(v))$.

Composition is inherited in the natural way from $\C$, as is the tensor: on objects this is simply the formal pairing $\_ \Box \_$, and given arrows $f\in \F_S(u,v)$, $g\in \F_S(x,y)$ we have 
\[ f\Box g = f\otimes g\in \F_S(u\Box x,  v\Box y) = \C(Inst(u)\otimes Inst(x) ,Inst(v)\otimes Inst(y)) . \]
The assignment $Inst:Tree\rightarrow Ob(\C)$ extends in a natural way to a strict semi-monoidal functor;  using computer science terminology, we call this the {\bf instantiation functor} $Inst_S: (\F_S,\Box)\rightarrow (\C,\otimes)$. It is as above on objects, and the identity on hom-sets, as $\F_S(u,v) = \C(Inst(u),Inst(v))$. It is immediate that this epic strict semi-monoidal functor is a semi-monoidal equivalence of categories. When the object is clear from the context, we simply write $Inst: (\F_S,\Box)\rightarrow (\C,\otimes)$

The image of $Inst_S$ is the full semi-monoidal subcategory  of $\C$ inductively generated by the object $S$, together with the tensor $\_ \otimes \_$. We refer to this as the semi-monoidal category {\bf generated by $S$ within $(\C,\otimes)$}, denoted $(\C_S,\otimes)$.
\end{definition}

Based on the above definitions, the following are immediate:

\begin{lemma}\label{monic-epic}
Let $(S,\code)$ be a self-similar structure of a semi-monoidal category $(\C,\otimes)$. Then the small semi-monoidal categories $(\W,\Box)$, $(\X,\Box)$ and $(\F_S,\Box)$ have the same set of objects, and 
\begin{enumerate}
\item The tuple $(x,\code)$ is a self-similar structure of $(\F_S,\Box)$.
\item The functors $\X Sub:(X,\Box)\rightarrow (\F_S,\Box)$ and $\W Sub(W,\Box) \rightarrow (F_S,\Box)$ are monic.
\item The following diagram commutes:
\[
\xymatrix{ (\W,\Box) \ar[rr]^{\W Sub} \ar[drr]|{\W Sub} 	&& (F_S,\Box) \ar[d]|{Inst}  && (\W,\Box) \ar[ll]_{\X Sub} \ar[dll]|{\X Sub}  \\ 
											&& (\C_S,\otimes) 		&&	\\
		}
\]
 \end{enumerate}
 \end{lemma}
 \begin{remark}
The commutativity of the left hand triangle in the above diagram is well-established, and part of a standard approach to coherence for associativity and other properties. In particular, Joyal \& Street phrased MacLane's theorem as an equivalence between the free monoidal category on a category and the free strict monoidal category on a category (see \cite{JS} for details and extensions of this approach).
 \end{remark}
 
 \subsection{Functors from categories to monoids}
The monic functor $\X Sub:(X,\Box)\rightarrow (\F_S,\Box)$ specifies a distinguished wide semi-monoidal subcategory of $(\F_S,\Box)$; we use the following notation and terminology for its arrows:
\begin{definition}\label{gencodedecode}
Given a self-similar structure $(S,\code)$ of a semi-monoidal category, we define an object-indexed family of arrows, the {\bf generalised code isomorphisms}, by  $\{ \code_u = \X Sub(x\leftarrow u) \in \F_S(u,x)\}_{u\in Ob(\F_S)}$.  We refer to their inverses, $\{ \decode_u= \X Sub(u\leftarrow x) \in \F_S(x,u) \}_{u\in Ob (\F_S)}$ as the {\bf generalised decode isomorphisms}.
\end{definition}

\begin{remark}\label{gencode_as_splittings}
As observed in Remark \ref{nat_trans} below, the above object-indexed families of arrows are the components of a natural transformation. An alternative perspective is that $(\code_u,\decode_u)$ is a splitting of $1_u$, and $(\decode_u,\code_u)$ is a splitting of $1_x$.  The unique isomorphism $\X Sub(v\leftarrow u)=\decode_v\code_u\in \F_S(u,v)$ is then the isomorphism exhibiting the uniqueness up to isomorphism of idempotent splittings described in Corollary \ref{unique}.
\end{remark}

\begin{remark}\label{inductive_def} As $(\F_S,\Box)$ is freely generated, we may give an inductive characterisation of the generalised code arrows by:
\begin{itemize}
\item $\code_x=1_S\in \C(S,S)=\F_S(x,x)$.
\item $\code_{x\Box x} = \code\in\C(S\otimes S,S)= \F_S(x\Box x,x)$
\item $\code_{u\Box v} = \code_{x\Box x}(\code_u\Box \code_v)\in \F_S(x,u\Box v)$.
\end{itemize}
and similarly for the generalised decode arrows. This is used as a definition in \cite{PH99}, where it is (incorrectly) assumed that $(C_S,\otimes) \cong (\F_S,\Box)$ in every case -- an assumption holds for the particular examples considered there, but not  more generally.\end{remark}
These generalised code / decode arrows allow us to define a fully faithful functor from $\F_S$ to the endomorphism monoid $\F_S(x,x)$, considered as a single-object category.
\begin{definition}\label{gencon}
Let $(S,\code)$ be a self-similar object of a semi-monoidal category $(\C,\otimes)$. We denote the endomorphism monoid $\F_S(x,x)$, considered as a single-object category, by $End(x)$, and define the {\bf generalised convolution functor} $\Phi_{\code} : \F_S\rightarrow End(x)$ by 
\begin{itemize}
\item {\bf Objects} $\Phi_{\code} (A)=x$, for all $A\in Ob(\F_S)$.
\item {\bf Arrows} Given $f\in \F_S(A,B)$, then $\Phi_{\code} (f)= \code_Bf\decode_A$, where $\code_\_$ and $\decode_\_$ are as in definition \ref{gencodedecode}.
\end{itemize}
When the self-similar structure in question is apparent from the context, we will omit the subscript and write $\Phi : \F_S\rightarrow End(x)$.
\end{definition}
\begin{proposition}The generalised convolution functor given above is indeed a fully faithful functor.
\end{proposition}
\begin{proof}
For all $a\in Ob(\F_S)$, $\Phi  (1_a) = 1_x$. Given $f\in \F_S(a,b)$, $g\in \F_S(b,c)$, then $\Phi (gf)= \code_c gf \decode_a = \code_c g \decode_b\code_b f \decode_a = \Phi  (g) \Phi  (f)$, and thus $\Phi $ is a functor. For all $h\in \F_S(x,x)$, $\Phi  (h)=h\in End(x)= \F_S(x,x)$, and thus $\Phi $ is full. Finally, given $f,f'\in \F_S(a,b)$, then
\[ \Phi  (f) =  \Phi  (f') \ \Leftrightarrow \ \code_bf\decode_a=\decode_b f'\code_a \ \Leftrightarrow \ \decode_b\code_b f \decode_a\code_a = \decode_b\code_b f' \decode_a\code_a \ \Leftrightarrow \ f=f' \]
and thus $\Phi :\F_S \rightarrow End(x)$ is faithful.
\end{proof}
A simple corollary is that the category freely generated by a self-similar object, and the endomorphism monoid of that object (considered a a single-object category) are equivalent:
\begin{corol}\label{equiv_cats}
Let $(S,\code)$ be a self-similar object of a semi-monoidal category $(\C,\otimes)$. Then the categories $\F_S$ and $End(x)$ are equivalent.
\end{corol}
\begin{proof}
Since $\Phi$ is fully faithful, it simply remains to prove that it is isomorphism-dense. For arbitrary $u\in Ob(\F_S)$, the generalised code/decode arrows $\code_u\in \F_S(u,x)$ and $\decode_u=\code_u^{-1}\in \F_S(x,u)$ exhibit the required isomorphism $End(x)\cong \F_S(x,x) \cong \F_S(u,u)$.
\end{proof}
In Definition \ref{stardef} below, we give a semi-monoidal tensor on $End(x)$ that makes the above equivalence a semi-monoidal equivalence of categories. 

\begin{remark}\label{nat_trans}
Corollary \ref{equiv_cats} guarantees the existence of suitable functors exhibiting this equivalence of categories; more explicitly, let us denote the obvious inclusion by $\iota:End(x)\hookrightarrow \F_S$. Then $\Phi\iota=Id_{End(x)}$, and  there is a natural transformation from $\iota\Phi$ to $Id_{\F_S}$ whose components are the generalised decode isomorphisms of Definition \ref{gencodedecode}:
\[ \xymatrix{
										& \F_S \ar[dr]^{\Phi}\ar@{}[d]|{commutes}	&				& \F_S 	\ar[rr]^{Id_{\F_S}} \ar[dr]_{\Phi}	&					& \F_S 	\\
 End(x) 	 \ar@{^(->}[ur]^{\iota}	\ar[rr]_{Id_{End(x)}}		&								& End(x)			& 								& End(x) \ar[ur]_{\iota} \ar@{=>}[u]_{ \code_{\_ }}	& 			\\						
}
\]
\end{remark}

It is also almost immediate that a diagram over $\F_S$ commutes iff its image under $\Phi$ commutes; we prove this explicitly in order to illustrate how this relies on uniqueness of generalised code / decode arrows:
\begin{corol}\label{abovebelow}
Let $(S,\code)$ be a self-similar object of a semi-monoidal category $(\C,\otimes)$. 
Then a diagram $\mathfrak D$ over $\F_S$ commutes iff $\Phi (\mathfrak D)$ commutes.
\end{corol}
\begin{proof}$ $ \\
$(\Rightarrow)$ This is a simple, well-known consequence of functoriality. 
\\
$(\Leftarrow)$
Let $\mathfrak D$ be an arbitrary diagram over $(\F_S,\Box)$. Up to the obvious inclusion $\iota:End(x)\hookrightarrow \F_S$, $\mathfrak D$ and $\Phi(\mathfrak D)$ are diagrams in the same category; we treat their disjoint union $\mathfrak D \uplus  \Phi_(\mathfrak D)$ as a single diagram. 
\begin{figure}[h]
\caption{$\mathfrak D$ and $\Phi  (\mathfrak D)$ as a single diagram}\label{augmented}
\begin{center}
\scalebox{0.9}{$ \xymatrix{ 
\ar@{}[d]|<<<<<*+[F-,]{\mathfrak D}		&											&& v \ar[drr]^{g}	 			&&		&&																				&& v \ar[drr]^{g}  \ar@<0.8ex>[dd]|<<<<<<<<{\  \code_v}			&&																						\\
								&u \ar[rrrr]_<<<<<<<<<<{h} \ar[urr]^f					&&						&& w		&&		u \ar[rrrr]_<<<<<<<<<<{h} \ar[urr]^f \ar@<0.8ex>[dd]|<<<<<<<<{\  \code_u}				&&													&& w	 \ar@<0.8ex>[dd]|<<<<<<<<{\  \code_w}											\\
\ar@{}[d]|<<<<<*+[F-,]{\Phi(\mathfrak D)}	&											&& x	 \ar[drr]^{\Phi(g)}		&& 		&&																				&& x	 \ar[drr]^{\Phi(g)}\ar@<0.8ex>[uu]|<<<<<<<<{ \decode_v}		&& 																						\\
								&x  \ar[rrrr]_<<<<<<<<<<{\Phi(h)} \ar[urr]^{\Phi(f )}		&&						&&  x		&& x  \ar[rrrr]_<<<<<<<<<<{\Phi(h)} \ar[urr]^{\Phi(f )} \ar@<0.8ex>[uu]|<<<<<<<<{ \decode_u}	&&													&&  x \ar@<0.8ex>[uu]|<<<<<<<<{ \decode_w}	
}
$}
\end{center}\end{figure}
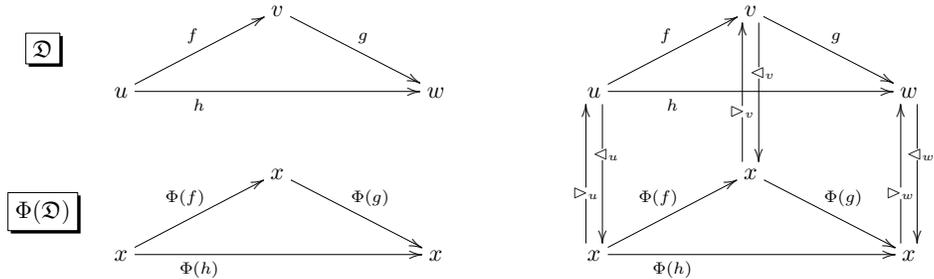
We then add edges to $\mathfrak D \uplus  \Phi(\mathfrak D)$ by linking each node $n$ with its image using the unique generalised code / decode arrows. This is illustrated in  Figure \ref{augmented}. 

Each additional polygon added to $\mathfrak D \uplus \Phi (\mathfrak D)$ commutes by definition of $\Phi$.  Thus the entire diagram commutes iff  $\mathfrak D$ commutes iff $\Phi (\mathfrak D)$ commutes. 
\end{proof}

\subsection{Semi-monoidal tensors on monoids}
We now exhibit a semi-monoidal tensor on the endomorphism monoid of a self-similar object such that the equivalence of Corollary \ref{equiv_cats} becomes a semi-monoidal equivalence.
\begin{lemma}\label{technical}
Let $(S,\code)$ be a self-similar structure of a semi-monoidal category 
and let $(\F_S,\Box,t_{\_,\_,\_})$ be the semi-monoidal category freely generated by $S$. Then, 
up to the inclusion $End(x)\hookrightarrow \F_S$,
\begin{enumerate}
\item $\Phi (f\Box  g) = \Phi \left( \Phi (f) \Box  \Phi (g) \right)$.
\item $\Phi(t_{u,v,w})=\Phi (t_{x,x,x})$, for all $u,v,w\in Ob(\F_S)$.
\end{enumerate}
\end{lemma}
\begin{proof}
$ $ \\
\begin{enumerate}
\item Given $f\in \F_S(a,b)$ and $g\in \F_S(u,v)$, then $ \Phi  (f\Box  g) = \code_{b\Box  v}(f\Box  g)\decode_{a\Box  u} = $
\[  \code_{x\Box x}(\code_b\Box  \code_v)((f\Box  g) (\decode_a\Box  \decode_u)\decode_{x\Box x} 
	=  \code_{x\Box x} ( \code_bf\decode_a \Box  \code_v g \decode_u ) \decode_{x\Box x} \]
\[ = \code_{x\Box x} \left( \Phi  (f) \Box  \Phi  (g) \right) \decode_{x\Box x} = \Phi  \left( \Phi  (f) \Box  \Phi  (g) \right) \]
\item By definition, $\Phi (t_{u,v,w})= \code_{(u\Box  v)\Box  v} t_{u,v,w} \decode_{u\Box  (v\Box  w)}$, and by Remark \ref{inductive_def}, the following diagram commutes:
\[ \xymatrix{
x \ar[rr]^{\decode_{x\Box x}} 	\ar[d]|{\Phi (t_{u,v,w})}	&& x\Box  x \ar[rr]^{1_x\Box \decode_{x\Box x}}		&& x\Box  (x\Box  x) \ar[rr]^{\decode_u\Box  (\decode_v\Box  \decode_w)}	&&		u\Box  (v\Box  w) \ar[d]|{t_{u,v,w}} \\
x												&& x\Box  x\ar[ll]^{\code_{x\Box x}} 				&& (x\Box  x) \Box  x \ar[ll]^{\code_{x\Box x}\Box  1_x} 					&& (u\Box  v)\Box  w \ar[ll]^{(\code_u \Box  \code_v)\Box \code _w} 
}
\]
By naturality of $t_{\_,\_,\_}$, 
\[ \begin{array}{rcl} 
\Phi (t_{u,v,w}) & = & \code_{x\Box x}(\code_{x\Box x}\Box 1_x) 
((\code_u \Box \code_v) \Box \code_w) 
((\decode_u \Box \decode_v)  \Box \decode_w)
t_{x,x,x}
(1_x \Box \decode_{x\Box x})
\decode_{x\Box x}  \\
 &   =   &  \code_{x \Box x}(\code_{x\Box x} \Box 1_x) t_{x,x,x} (1_x \Box  \decode_{x\Box x} ) \decode_{x\Box x}  \\
  & =  &  \Phi (t_{x,x,x}) \\
\end{array} \]
%
as required. 
\end{enumerate}
\end{proof}
Based on the above lemma, we give a semi-monoidal tensor on the endomorphism monoid of a self-similar object.
\begin{definition}\label{stardef}
Let $(S,\code)$ be a self-similar object of a semi-monoidal category. 
We define the {\bf semi-monoidal tensor induced by $(S,\code)$} to be the monoid homomorphism
$\_ \star_\code \_ : End(x)\times End(x)\rightarrow End(x)$ given by 
\[ f \star_\code g \ \stackrel{def.}{=} \ \Phi(f\Box g) \ =\  \code_{x\Box x} (f\Box g)\decode_{x\Box x}  \]
When the self-similar structure is clear from the context, we elide the subscript, and write $\_ \star \_ : End(x)\times End(x)\rightarrow End(x)$
\end{definition}

\begin{theorem}\label{untyped}
The operation $\_ \star \_$ defined above is a semi-monoidal tensor, and thus $(End(x), \_ \star \_)$ is a semi-monoidal monoid.
\end{theorem}
\begin{proof}
Functoriality of $\Phi$ gives $1\star 1 =\Phi(1_{x}\Box 1_{x})=\Phi(1_{x\Box x})=1$, and  
\[ (f\star g)(h\star k) = \Phi(f\Box g)\Phi(h\Box k)=\Phi((f\Box g)(h\Box k))=\Phi(fh\Box gk)=fh\star gk \]
Let us define $\alpha=\Phi(t_{x,x,x})$. From Part 2. of Lemma \ref{technical}, $\Phi(t_{u,v,w})=\alpha$, for arbitrary $u,v,w\in Ob(\F_S)$. From Part 1. of Lemma \ref{technical}, and  Corollary \ref{abovebelow}, 
\[ \alpha (f\star(g\star h)) = ((f\star g)\star h)\alpha \ \ \forall f,g,h\in End(x) \]
and $\alpha^2 = (\alpha \star 1)\alpha (1\star \alpha)$. Thus both naturality and MacLane's pentagon condition are also satisfied.
\end{proof}

\begin{corol}\label{monoidal_phi}
The functor $\Phi:\F_S\rightarrow End(x)$ satisfies $\Phi(f\Box g )= \Phi(f) \star \Phi(g)$, and thus is a strict semi-monoidal functor $\Phi:(\F_S,\Box)\rightarrow (\C_S,\otimes)$.
\end{corol}
\begin{proof}
This follows from Theorem \ref{untyped}  and Theorem \ref{technical}.
\end{proof}
At a given self-similar object $S\in Ob (\C)$, the semi-monoidal tensor is determined by the choice of isomorphism $\code\in \C(S\otimes S,S)$; however, these are related by conjugation in the obvious way, and thus $\_ \star \_$ is unique up to unique isomorphism.

\begin{proposition}\label{Drelated}
Let $(S,c)$ and $(S,\code)$ be self-similar structures at a given self-similar object. Then $f \star_{\code} g = \code c^{-1} (f\star_c g) c \decode$ for all $f,g\in End(x)$.
\end{proposition}
\begin{proof}
This follows by direct calculation on Definition \ref{stardef}; alternatively, and more structurally, it follows from the uniqueness up to unique isomorphism of idempotent splittings, and hence self-similar structures (Corollary \ref{unique}). \end{proof}

\begin{remark}
The above Proposition does not imply that all semi-monoidal tensors on a given monoid are related by conjugation. As a counterexample, the monoid of functions on $\mathbb N$ has distinct semi-monoidal tensors, arising from the fact that it is a self-similar object in both $(Fun, \times)$ and $(Fun,\uplus)$, that are clearly not related in this way (the relationship between the two is non-trivial and a key part of Girard's Geometry of Interaction program \cite{GOI,GOI2}, the details of which are beyond the scope of this paper).
\end{remark}

From Theorem \ref{no_sim_strict}, $(End(x), \_ \star \_)$ can only be strictly associative when $x$ is the unit object for $\_ \star \_$. When $(S,\code)$ is a self-similar structure of a strictly associative semi-monoidal category (e.g. the rings isomorphic to their matrix rings characterised in \cite{HL}), the associativity isomorphism for $\_ \star \_$ has the following neat form:

\begin{proposition}\label{not_strict}
Let $(S,\code)$ be a self-similar object of a strictly associative semi-monoidal category $(C,\otimes)$. Then the associativity isomorphism for $(End(x),\star)$ is given by $\alpha=\code(\code\otimes 1_S)(1_S\otimes \decode)\decode\in End(x)\cong \C(S,S)$.
\end{proposition} 
\begin{proof}
This follows by direct calculation on Part 2. of Lemma \ref{technical}.
\end{proof}

\subsection{The strictly self-similar form of a monogenic category}
The following is now immediate:
\begin{theorem}\label{3equiv} Let $(S,\code)$ be a self-similar structure of a semi-monoidal category $(\C,\otimes)$. Then $(\C_S,\otimes)$, $(\F_S,\Box)$ and $(End(x),\star)$ are semi-monoidally equivalent.
\end{theorem}
\begin{proof}
The semi-monoidal equivalence between $(\F_S,\Box)$ and $(\C_S,\otimes)$ is given by the semi-monoidal functor of Definition \ref{ins_def}. From Corollary \ref{monoidal_phi}, the equivalence of categories between $\F_S$ and $End(x)$ gives a semi-monoidal equivalence between $(\F_S,\Box)$ and $(End(x),\star)$.
\end{proof}

\begin{corol}Every monogenic semi-monoidal category with a self-similar generating object is semi-monoidally equivalent to a semi-monoidal monoid. This justifies the description of $(End(x),\star)$ as the {\bf self-similarity strictification} of $(\C_S,\otimes)$. 
\end{corol}

\begin{remark}
A general principle is that `categorical structures' are preserved by equivalences of categories. For example, if $\C_S$ is closed, then so is $End(x)$; this is used implicitly in \cite{LS} to construct single-object analogues of Cartesian closed categories, and in \cite{PH98,PH99} to construct single-object analogues of compact closure. Similarly, when $(C_S,\otimes)$ admits projections / injections, $(End(x),\star)$ contains a copy of Girard's dynamical algebra \cite{PH98,MVL,PH99} and under relatively light additional assumptions, admits a matrix calculus \cite{PH14a}. In general, we may find single-object (i.e. monoid) analogues of a range of categorical properties. 
\end{remark}

We may now answer the question posed in Remark \ref{doesitfactor}.

\begin{corol}\label{inclusion_kills}
The diagram of Figure \ref{killer_inclusion} commutes precisely when the self-similar object in question is the unit object.
\end{corol}
\begin{proof} 
In the self-similarity strictification of $(\C_S,\otimes)$, the self-similarity is exhibited by identity arrows. Commutativity of the diagram of Figure \ref{killer_inclusion} implies that $(End(x),\star)$ has a strictly associative semi-monoidal tensor, so by Theorem \ref{no_sim_strict}, the unique object of $End(x)$ is the unit object for $\_ \star \_$. The equivalences of Theorem \ref{3equiv} then imply that $S$ is the unit object for $(\C_S,\otimes)$.
\end{proof}

Self-similarity strictification also illustrates a close connection between the generalised convolution and instantiation functors; informally, generalised convolution is simply instantiation in an isomorphic category:
\begin{proposition}\label{Kfunctor}
Let $(S,\code)$ be a self-similar object of a semi-monoidal category $(\C,\otimes )$, and denote by $(\F_S,\Box)$ and $(\F_x,\Diamond)$ the semi-monoidal categories freely generated by $S$, and the unique object of $(End(x),\star)$, respectively. Then there exists a semi-monoidal isomorphism $K: (F_S,\Box)\rightarrow (\F_x,\Diamond)$ such that the following diagram of semi-monoidal categories commutes:
\[ \xymatrix{
(\F_S,\Box) \ar@/^8pt/[rr]|K \ar[dr]_{\Phi}	&			& (\F_x,\Diamond) \ar[dl]^{Inst} \ar@/^8pt/[ll]|{K^{-1}} \\
							& (End(x),\star)	&						\\
}
\] 
\end{proposition}
\begin{proof}
We define the semi-monoidal functor $K:(\F_S,\Box)\rightarrow (\F_x,\Diamond)$ as follows:
\begin{itemize}
\item The small categories $(\F_x,\Diamond)$ and $(\F_S,\Box)$ have the same underlying set of objects; we take $K$ to be the identity on objects.
\item Given $f\in \F_S(u,v)$ we define $K(f)=\code_v f \decode_u \in \F_S(x,x)=\F_x(u,v)$.
\end{itemize}
The inverse is immediate, as is the (strict) preservation of the semi-monoidal tensor.  The commutativity of the above diagram follows by expanding out the definitions of $\Phi$ and $Inst$.
\end{proof}

\section{General coherence for self-similarity}\label{full_coherence}
We now consider coherence in the general case. Let  us fix a a self-similar structure $(S,\code)$ of a semi-monoidal category $(\C,\otimes ,\tau_{\_,\_,\_})$. We will abuse notation slightly; based on the monoid isomorphism $End(x)\cong \C(S,S)$, we treat the semi-monoidal tensor $\_ \star_\code \_$ equally as an operation on $\C(S,S)=\F_S(x,x)$ and denote the (unique) associativity isomorphism for $\_\star \_$ as $\alpha\in \C(S,S)$.    

The question we address is the following:\\
 
 {\em 
 \noindent
 \hspace{6em}Given a diagram over $\mathcal C_S$ with arrows built inductively from  
 \[ \{ \ \_ \otimes \_ \ ,  \tau_{\_,\_,\_} \ ,\  \code \ , \  \_\star\_ \ ,\ \ \alpha \ ,\  ( \ )^{-1}  \ \}, \]  
 \hspace{6em}when may it be guaranteed to commute?\\  }

We first fix some terminology.
\begin{notation} Given a category $\C$ and a class  $\Gamma$ of operations and arrows of $\C$, we say that a diagram is {\bf canonical for $\Gamma$} when its edges are built inductively from members of $\Gamma$.  For example, in a semi-monoidal category $(\C,\_ \otimes \_ , \tau_{\_,\_,\_})$, a diagram canonical for $\{ \_ \otimes \_ , \tau_{\_,\_,\_} , (\ )^{-1} \}$ is a diagram canonical for associativity, as usually understood.
\end{notation}

The following demonstrates that a simple appeal to freeness is not sufficient:

\begin{proposition} 
Let $(\F_S,\Box , t_{\_,\_,\_} )$ be the semi-monoidal category freely generated by $S$.  Then, over $\F_S$ 
\begin{enumerate}
\item All diagrams canonical for $\{ \_ \Box \_  , t_{\_,\_\_ ,\_ }  , (\ )^{-1} \}$ commute.
\item All diagrams canonical for $\{ \_ \Box \_  , \code_\_ , , (\ )^{-1}  \}$ commute.
\item All diagrams canonical for $\{ \_ \Box \_  , t_{\_,\_\_ ,\_ } ,  \code_\_ , , (\ )^{-1}  \}$ commute iff $S$ is a unit object for $(\C_S,\otimes)$.
\end{enumerate}
\end{proposition}
\begin{proof}
1. is well-established; it follows from the monic-epic decomposition of MacLane's substitution functor described in Lemma \ref{monic-epic} and (in the monoidal case) is commonly used \cite{JS} to study coherence. 2. follows similarly from Lemma \ref{monic-epic}. For 3., the following diagram is canonical for self-similarity and associativity:
\[ \xymatrix{ 
				&	 \ar[dl] _{\decode_{x\Box x}\Box 1_x}   x\Box x 	 \ar[dr] ^{1_x\Box \decode_{x\Box x}}	&	\\
(x\Box x)\Box x		&																		& x\Box (x\Box x) \ar[ll]^{t_{x,x,x}} 
}
\]
Applying $\Phi:(F_S,\Box)\rightarrow (End(x),\star)$ to this diagram gives the associativity isomorphism for $(End(x),\star )$ as $\alpha=1_x$, so by Theorem \ref{no_sim_strict} the unique object of $End(x)$ is the unit object for $\_ \star \_$. Appealing to the semi-monoidal equivalences of Theorem \ref{3equiv} gives that $S\in Ob(\C_S)$ is the unit object for $\_ \otimes \_$.
\end{proof}

We now introduce an equivalence relation on diagrams over $\F_S$ that allows us to answer this question in the free case:

\begin{definition}
As $\mathcal F_S$ is a small category, we may treat a diagram $\mathfrak D$ over $\F_S$, with underlying directed graph $G=(V,E)$, as a pair of functions $\mathfrak D_V:V\rightarrow Ob (\F_S)$ and $\mathfrak D_E:E\rightarrow Arr(\F_S)$ satisfying
\[ \mathfrak D_E(e) \in \F_S(\mathfrak D_V(v),\mathfrak D_V(w)) \ \ \mbox{ for all edges } \  \ v\stackrel{e}{\longrightarrow} w \ \in E \]
We will omit the subscripts on $\mathfrak D_V$ and $\mathfrak D_E$  when the context is clear.

Given diagrams $\mathfrak T,\mathfrak U$ with underlying graphs $G=(V,E)$ and $G'=(V',E')$ respectively, we say they are {\bf self-similarity equivalent}, written $\mathfrak T \sim_\dc \mathfrak U$ when there exists a graph isomorphism $\eta:G\rightarrow G'$ such that, for all edges $s\stackrel{e}{\longrightarrow} t$ of $G$, the following diagram commutes:
\[ \xymatrix{
\mathfrak T(s) \ar[rr]^{\mathfrak T(e)} 		&& \mathfrak T(t) \\
\mathfrak U(\eta(s)) \ar[rr]_{\mathfrak U (\eta(e))} 	\ar[u]|{\decode_{\mathfrak T(s)}\code_{\mathfrak U(\eta(s))}}	&& \mathfrak U(\eta(t)) \ar[u]|{\decode_{\mathfrak T(t)}\code_{\mathfrak U(\eta(t))}} \\
}
\]
(An intuitive description is illustrated by example in Figure \ref{verticals}). 
This is an equivalence relation, since the object-indexed isomorphisms $\code_\_$ and $\decode_\_$ specify a wide posetal subcategory of $\F_S$. We denote the corresponding equivalence classes by $[\_ ]_\dc$.
\end{definition} 

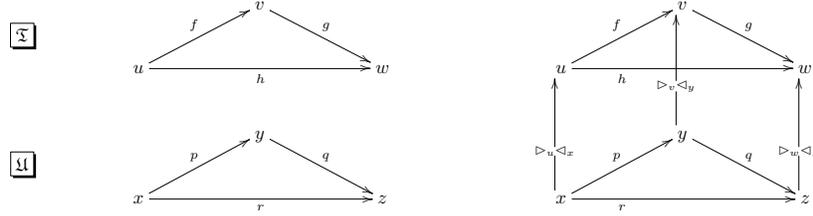
\begin{figure}
\caption{$\mathfrak T \sim_\dc \mathfrak U$ when all `vertical' squares commute in the rhs diagram}\label{verticals}
\begin{center}
\scalebox{0.7}{$ \xymatrix{ 
\ar@{}[d]|<<<<<*+[F-,]{\mathfrak T}	&& 					&& v	\ar[drr]^g	&&		&	&&																		&& v \ar[drr]^{g}  										&&																						\\
							&& u \ar[urr]^f	\ar[rrrr]_h	&&			&& w		&	&&		u \ar[rrrr]_<<<<<<<<<<{h} \ar[urr]^f 										&&													&& w	 											\\
\ar@{}[d]|<<<<<*+[F-,]{\mathfrak U}	&&					&& y	\ar[drr]^q	&&		&	&&																		&& y	 \ar[drr]^{q }\ar@<0.8ex>[uu]|<<<<<<<<{ \decode_v\code_y}	&& 																						\\
							&&x	\ar[urr]^p	\ar[rrrr]_r 	&&			&& z		&	&& x  \ar[rrrr]_<<<<<<<<<<{r} \ar[urr]^{p} \ar@<0.8ex>[uu]|<<<<<<<<{ \decode_u\code_x}	&&													&&  z \ar@<0.8ex>[uu]|<<<<<<<<{ \decode_w\code_z}	
}
$}
\end{center}
\end{figure}

\begin{lemma}\label{simcomm}
Given $\mathfrak T \sim_\dc \mathfrak U$ over $\F_S$, then $\mathfrak T$ commutes iff $\mathfrak U$ commutes.
\end{lemma}
\begin{proof}
This is immediate from the definition, and a slight generalisation of the reasoning in the proof of Theorem \ref{abovebelow}.
\end{proof}
We may now demonstrate commutativity for a class of diagrams in the free setting:
\begin{theorem}\label{formalcase}
Let $(S,\code)$ be a self-similar structure of a semi-monoidal category $(\C,\otimes)$, and let $\mathfrak D$ be a diagram over $(\F_S,\Box)$ canonical for $\{ \Box , t_{\_,\_,\_} , \code_\_ ,\star , \alpha , (\ )^{-1} \}$. The following two conditions are equivalent, and characterise a class of diagrams guaranteed to commute:
\begin{enumerate}
\item $[ \mathfrak D]_\dc$ contains a diagram canonical for $\{ \Box , t_{\_,\_,\_} , (\ )^{-1} \}$.
\item $\Phi(\mathfrak D)$, which is canonical for $\{ \star ,\alpha , (\ )^{-1} \}$, is guaranteed to commute by MacLane's coherence theorem for associativity.
\end{enumerate}
\end{theorem}
\begin{proof}  
Somewhat redundantly, we give separate proofs that both these conditions characterise a class of commuting diagrams:
\begin{enumerate}
\item From the basic theory of coherence for associativity  (see Lemma \ref{monic-epic}), in $(\F_S,\Box)$, all diagrams canonical for $\{ \Box , t_{\_,\_,\_} , (\ )^{-1} \}$ are guaranteed to commute by MacLane's theorem; our result then follows from Lemma \ref{simcomm}. 
\item  By construction, $\Phi(\code_\_) = 1_x$ and $\Phi(f\Box g)=\Phi(f)\star \Phi(g)$ for all arrows $f,g$. Thus $\Phi(\mathfrak D)$ is indeed canonical for $\{ \star ,\alpha , (\ )^{-1} \}$. From Corollary \ref{abovebelow}, $\Phi(\mathfrak D)$ commutes iff $\mathfrak D$ commutes. Thus when MacLane's theorem predicts $\Phi(\mathfrak D)$ to commute, $\mathfrak D$ also commutes.
\end{enumerate}
We now show that they characterise the same class of commuting diagrams:\\
$(1.\Rightarrow 2.)$ Let $\mathfrak R \in [\mathfrak D]_\dc$ be canonical for $\{ \_ \Box \_ ,t, (\ )^{-1} \}$. As the functor $K:(\F_S,\Box, t_{\_,\_,\_})\rightarrow (\F_x,\Diamond,t'_{\_,\_,\_})$ of Proposition \ref{Kfunctor} is semi-monoidal, $K(\mathfrak R)$ is canonical for $\{ \Diamond, t'_{\_,\_,\_}, ( \ )^{-1} \}$ and thus $Inst (K(\mathfrak R)) = \Phi (\mathfrak R)$ is predicted to commute by MacLane's theorem. However, up to isomorphism of the underlying graph, $\Phi(\mathfrak R)$ is identical to $\Phi(\mathfrak D)$.
\\
$(2.\Rightarrow 1.)$ As $\Phi(\mathfrak D)$ is predicted to commute by MacLane's theorem, there exists some diagram $\mathfrak P$ over $(\F_x,\Diamond , t'_{\_,\_,\_})$ that is canonical for $\{ \Diamond , t'_{\_,\_,\_} , (\ )^{-1} \}$ satisfying $Inst(\mathfrak P) =  \Phi(\mathfrak D)$. As the isomorphism $K:(\F_S,\Box, t_{\_,\_,\_})\rightarrow (\F_x,\Diamond,t'_{\_,\_,\_})$ of Proposition \ref{Kfunctor} is semi-monoidal, $K^{-1}(\mathfrak P)$ is canonical for $\{ \Box , t_{\_,\_,\_}, (\ )^{-1} \}$ and by construction $K^{-1}(\mathfrak P) \sim_\dc \mathfrak D$. Thus our result follows.
\end{proof}

The above theorem answers the question posed at the start of this section in the `formal' setting $(\F_S,\Box)$. To map this free setting to the concrete setting, we apply the $Inst:(\F_S,\Box)\rightarrow (\C_S,\otimes)$ functor, giving the following corollary:

\begin{corol}\label{concretesetting}
Given a diagram $\mathfrak E$ over $\C_S$ canonical for $\{ \otimes  ,  \tau_{\_,\_,\_}  ,  \code ,  \star , \alpha  ,  ( \ )^{-1} \}$, then $\mathfrak E$ is guaranteed to commute when there exists a diagram $\mathfrak D$ over $\F_S$ that is canonical for $\{ \Box , t_{\_,\_,\_} , \code_\_ ,\star , \alpha , (\ )^{-1} \}$ satisfying
\begin{enumerate}
\item $\mathfrak D$ is guaranteed to commute by Theorem \ref{formalcase} above. 
\item $Inst(\mathfrak D)=\mathfrak E$.
\end{enumerate}
\end{corol}

The identification of generalised convolution as the instantiation functor of a semi-monoidally isomorphic category (Proposition \ref{Kfunctor}) then translates the above into the following neat heuristic for characterising such diagrams:
\begin{corol}
Let $\mathfrak E$ be a diagram over $\C_S$ canonical for $\{ \otimes   ,  \tau_{\_,\_,\_}  ,  \code ,   \star , \alpha  ,  ( \ )^{-1} \}$. Let us  form a new diagram $\mathfrak E^\flat$ over $C(S,S)$  by the following procedure:
\begin{itemize}
\item Replace every object in $\mathfrak E$ by $S$.
\item Replace every occurrence of $\otimes$ by $\star$.
\item Replace every occurrence of $\tau_{\_,\_,\_}$ by $\alpha$.
\item Replace every occurrence of $\code$ by $1_S$.
\end{itemize}
Then $\mathfrak E^\flat$, which is canonical for $\{ \_ \star \_ , \alpha , (\ )^{-1} \}$,  is guaranteed to commute by MacLane's coherence theorem for associativity iff $\mathfrak E$ is guaranteed to commute by Corollary \ref{concretesetting}. 
\end{corol}


\end{document}